\documentclass[10pt]{amsart}

\usepackage{amsmath,amssymb,latexsym,enumerate}

\usepackage[margin=1in]{geometry}

\usepackage[svgnames]{xcolor}
\usepackage{polyhedra}

\newcommand{\C}{\mathbb{C}}

\newcommand{\bl}{\mathbf{l}}
\newcommand{\bu}{\mathbf{u}}

\newcommand{\mH}{\mathcal{H}}
\newcommand{\integers}{\mathbb{Z}}

\newcommand{\reals}{\mathbb{R}}
\newcommand{\chim}{\mathrm{Chim}}

\newcommand{\des}{{\rm des}\,}

\newcommand{\la}{\lambda}

\newcommand{\cone}{\mathrm{cone}}
\newcommand{\vspan}{\mathrm{span}}

\newcommand{\note}[1]{\par \smallskip\noindent
  \framebox{\begin{minipage}[c]{0.95 \textwidth} \ttfamily NOTE:
      #1 \end{minipage}}\par\smallskip}

\newtheorem{theorem}{Theorem}[section]
\newtheorem{lemma}[theorem]{Lemma}

\newtheorem{conjecture}[theorem]{Conjecture}

\theoremstyle{definition}

\newtheorem{remark}[theorem]{Remark}

\theoremstyle{remark}

\title[Generating functions and triangulations for lecture hall cones]{Generating functions and triangulations\\ for lecture hall cones}

\author[M. Beck]{
Matthias Beck
}
\address{Department of Mathematics\\
         San Francisco State University\\
         San Francisco, CA 94132, U.S.A.}
\email{mattbeck@sfsu.edu}

\author[B. Braun]{
Benjamin Braun
}
\address{Department of Mathematics\\
         University of Kentucky\\
         Lexington, KY 40506--0027, U.S.A.}
\email{benjamin.braun@uky.edu}

\author[M. K\"oppe]{
Matthias K\"oppe
}
\address{Department of Mathematics\\
         University of California\\
         Davis, CA 95616, U.S.A.}
\email{mkoeppe@math.ucdavis.edu}

\author[C. Savage]{
Carla D. Savage
}
\address{Department of Computer Science\\
         North Carolina State University \\
         Raleigh, NC 27695-8206, U.S.A.}
\email{savage@ncsu.edu}
 
\author[Z. Zafeirakopoulos]{
 Zafeirakis Zafeirakopoulos
}
\address{
Institute of Information Technologies\\
Gebze Technical University\\
Kocaeli, Turkey}
\email{zafeirakopoulos@gtu.edu.tr}

\thanks{
Thanks to Christian Haase for directing our attention to the existence of chimney polytopes and their triangulations, and to the anonymous referees for their helpful suggestions.
Matthias Beck was partially supported by grant DMS-1162638 of the U.S.\ National Science Foundation.
Benjamin Braun was partially supported by grant H98230-13-1-0240 of the U.S.\ National Security Agency.
Matthias K\"oppe was partially supported by grant DMS-0914873 of the U.S.\ National Science Foundation.
Carla Savage was partially supported by grant \# 244963 from the Simons Foundation.
Zafeirakis Zafeirakopoulos was partially supported by the strategic program \emph{Innovatives O\"O 2010 plus} by the Upper 
Austrian Government and by the Austrian Science Fund (FWF) grant W1214-N15 (project DK6) and
special research group Algorithmic and Enumerative Combinatorics SFB F50-06.
The authors thank the American Institute of Mathematics for support of our 
SQuaRE working group \emph{Polyhedral Geometry and Partition Theory.}
}

\begin{document}

\begin{abstract}
We investigate the arithmetic-geometric structure of the lecture hall cone
\[
L_n \ := \ \left\{\lambda\in \reals^n: \, 0\leq \frac{\lambda_1}{1}\leq \frac{\lambda_2}{2}\leq \frac{\lambda_3}{3}\leq \cdots \leq \frac{\lambda_n}{n}\right\} .
\]
We show that $L_n$ is isomorphic to the cone over the lattice pyramid of a reflexive simplex whose Ehrhart $h^*$-polynomial is given by the $(n-1)$st Eulerian polynomial, and prove that lecture hall cones admit regular, flag, unimodular triangulations.  
After explicitly describing the Hilbert basis for $L_n$, we conclude with observations and a conjecture regarding the structure of unimodular triangulations of $L_n$, including connections between enumerative and algebraic properties of $L_n$ and cones over unit cubes.
\end{abstract}

\maketitle

\section{Introduction}

For any subset $K\subset \reals^n$, we define the \emph{integer point transform} of $K$ to be the formal power series
\[
\sigma_K(z_1,\ldots,z_n) \ :=\sum_{m\in K \cap \integers^n}z_1^{m_1}\cdots z_n^{m_n} .
\]
Given a pointed rational cone $C:=\{x\in \reals^n: \, A \, x \ge 0 \}$, 
it is well known that $\sigma_C(z_1,\ldots,z_n)$ is  a rational function~\cite{BeckRobinsCCD}.
(Here \emph{pointed} requires that $C$ does not contain any nontrivial linear subspace of $\reals^n$ and
\emph{rational} requires that $A$ have rational---or equivalently, integral---entries.)
For $n\in \integers_{\geq 1}$, the \emph{lecture hall cone} is
\[
L_n \ := \ \left\{\lambda\in \reals^n: \, 0\leq \frac{\lambda_1}{1}\leq \frac{\lambda_2}{2}\leq \frac{\lambda_3}{3}\leq \cdots \leq \frac{\lambda_n}{n}\right\} .
\]
The elements of $L_n\cap \integers^n$ are called \emph{lecture hall partitions}, and the generating functions
$\sigma_{L_n}(z_1, z_2, \ldots,z_n)$ and $\sigma_{L_n}(q,q,\ldots,q)$ have been the subject of active research~\cite{sLectHallGorenstein,BME1,osheajones,BrightSavage,CLS2005,pensylsavage,savageschuster} since the discovery of the following  surprising result.
\begin{theorem}[Bousquet-M\'{e}lou \& Eriksson \cite{BME1}]
For $n\geq 1$, 
\[
\sigma_{L_n}(q,q,\ldots,q) \ = \ \frac{1}{\prod_{j=1}^n(1-q^{2j-1})} \, .
\]
\end{theorem}

Our focus in this note is on the arithmetic-geometric structure of the lecture hall cone, continuing our work
in \cite{sLectHallGorenstein}.
After giving the necessary background and terminology in Section~\ref{sec:background}, we prove in
Section~\ref{sec:eulerianreflexive} that $L_n$ is isomorphic to the cone over the lattice pyramid of a reflexive simplex whose Ehrhart $h^*$-polynomial is given by the $(n-1)$st Eulerian polynomial.
In Section~\ref{sec:triangulation} we prove that the lecture hall cones admit regular, flag, unimodular
triangulations.  In Section~\ref{sec:hilbert} we explicitly describe the Hilbert basis for $L_n$.
We conclude in Section~\ref{sec:outlook} with observations and a conjecture regarding the structure of unimodular triangulations of $L_n$, including connections between enumerative and algebraic properties of $L_n$ and cones over unit cubes.


\section{Background}\label{sec:background}

This section contains the necessary terminology and background literature to understand our results; it can
safely be skipped by the experts.

\subsection{Eulerian polynomials}
We recall the \emph{descent statistic} for $\pi \in S_n$
\[
\des(\pi) \ := \ |\{i: \, 1\leq i\leq n-1, \ \pi_i> \pi_{i+1}\}| \, ,
\]
which is encoded in the \emph{Eulerian polynomial}
\[
A_n(z) \ := \ \sum_{\pi\in S_{n}}z^{\des(\pi)} .
\]
An alternative definition of the Eulerian polynomial is via
\[
\sum_{t\geq 0}(1+t)^nz^t \ = \ \frac{A_{n}(z)}{(1-z)^{n+1}} \, .
\]

\subsection{Lattice polytopes}
A \emph{lattice polytope} is the convex hull of finitely many integer vectors in $\reals^n$.
The \emph{Ehrhart polynomial} of a lattice polytope $P$ is 
\[
i(P,t) \ := \ \left| tP \cap \integers^n \right|
\]
and the \emph{Ehrhart series} of $P$ is $1 + \sum_{t \ge 1} i(P,t) \, z^t$.
Ehrhart's theorem \cite{Ehrhart} asserts that $i(P,t)$ is indeed a polynomial; equivalently, the Ehrhart
series evaluates to a rational function of the form
\[
1 + \sum_{t \ge 1} i(P,t) \, z^t \ = \ \frac{h^*(z)}{(1-z)^{\dim(P)+1}},
\]
where we write $h^*(z) = \sum_{j=0}^{\dim(P)}h^*_j(P) \, z^j$.
By a theorem of Stanley \cite{StanleyDecompositions}, the $h^*_j(P)$ are nonnegative integers.
The polynomial $h^*(z)$ is called the \emph{Ehrhart $h^*$-polynomial} of~$P$.
Ehrhart polynomials and series have far-reaching applications in combinatorics, number theory, and beyond
(see, e.g., \cite{BeckRobinsCCD} for more).

A lattice polytope is \emph{reflexive} if its polar dual is also a lattice polytope.
A theorem of Hibi \cite{HibiDualPolytopes} says that $P$ is a lattice translate of a reflexive polytope if and only if the vector $(h_0^*(P),\ldots,h^*_{\dim(P)}(P))$ is symmetric.
Reflexive polytopes are of considerable interest in combinatorics, algebraic geometry, and theoretical
physics~\cite{BatyrevDualPolyhedra}.

\subsection{Triangulations}
A \emph{triangulation} of a lattice polytope $P$ is a collection of simplices that only meet in faces and whose
union is $P$.
A triangulation is \emph{unimodular} if all of its maximal simplices are. (A \emph{unimodular} simplex is the
convex hull of some $v_0, v_1, \dots, v_n \in \integers^n$ such that $\{ v_1 - v_0, v_2 - v_0, \dots, v_n - v_0 \}$ is a lattice basis of $\integers^n$.)
A triangulation of $P$ is \emph{regular} if there is a convex function $P \to \reals$ whose domains of
linearity are exactly the maximal simplices in the triangulation.
A triangulation is \emph{flag} if, viewed as a simplicial complex, its minimal non-faces are pairs of vertices.
(See, e.g., \cite[Section~1]{regulartriangulations} for more details.)

\newcommand\Hilb{H}
\newcommand\hilb{h}

\subsection{Gradings of cones}
Let $C \subset \reals^n$ be a pointed, rational $n$-dimensional cone. A \emph{grading} of $C$ is a vector $a \in
\integers^n$ such that $a \cdot p > 0$ for all $p \in C \setminus \{0\}$. With such a grading $a \in \integers^n$, we
associate its \emph{Hilbert function}
\[
    \hilb_C^a(t) \ := \ \left| \left\{ m \in C \cap \integers^n : \, a \cdot m = t \right\} \right| .
\]
Since $C$ is pointed, $\hilb_C^a(t) < \infty$ for all $t \in \integers_{\ge 1}$ and thus we can define the
\emph{Hilbert series}
\[
    \Hilb_C^a(z) \ := \ 1 + \sum_{t \ge 1} \hilb_C^a(t) \, z^t .
\]
It is naturally connected to the integer transform of $C$ via 
\[
\Hilb_C^a(z) \ = \ \sigma_C(z^{a_1},\dots,z^{a_d}) \, .
\]
For example, $\sigma_{L_n}(q, q, \dots, q)$ is the Hilbert series of the lecture hall cone $L_n$ with respect to the grading $(1, 1, \dots, 1)$, and the Ehrhart series of $P$ is the Hilbert series of the \emph{cone over} $P$,
\[
\cone(P) \ := \ \vspan_{\reals_{\ge 0}}\{(1,p): \, p\in P \} \, ,
\]
with respect to the grading $(1, 0, \dots, 0)$.


\section{Eulerian Gradings and Reflexive Simplices}\label{sec:eulerianreflexive}

The \emph{lecture hall polytope} was defined in \cite{savageschuster} as 
\[
P_n \ := \ \left\{ \la \in L_n : \, \la_n \leq n \right\} .
\]
From \cite[Corollary~2]{CLS2005} we know that $|tP_n \cap \integers^n| = (t+1)^n$, which means that the
Ehrhart series of $P_n$ is $\sum_{t \geq 0}(t+1)^n z^t$; so the Eulerian polynomial $A_n(z)$ is the $h^*$-polynomial of~$P_n$.

We are interested in associating a reflexive polytope with the lecture hall cone $L_n$.
The coefficients of an Eulerian polynomial are symmetric.
Note, however that $P_n$ has dimension $n$ whereas $A_n$ has degree $n-1$, and so $P_n$ is not reflexive.
On the other hand, Savage and Schuster~\cite[Corollary 2(b) \& Lemma 1]{savageschuster} proved that
\[
\sum_{\lambda\in L_n\cap \integers^n} z^{ \left\lceil \frac{ \lambda_n } n \right\rceil } \ = \ \frac{ A_n(z) }{ (1-z)^n } \, ,
\]
which is almost a Hilbert series (the exponent vector on the left-hand side does not correspond to a grading).
Instead we consider the grading $(0, \dots, 0, -1, 1)$, which defines $\deg(\lambda):=\lambda_n-\lambda_{n-1}$
for $\lambda=(\lambda_1,\ldots, \lambda_n)\in L_n$.

\begin{theorem}\label{thm:lecthalldes}
For all $n\geq 1$,
\[
\sigma_{L_n}(1,1,\ldots,1,z^{-1},z) 
\ = \ \sum_{\lambda\in L_n\cap \integers^n}z^{\deg(\lambda)}
\ = \ \frac{A_{n-1}(z)}{(1-z)^{n}} \, .
\]
\end{theorem}

\begin{proof}
Fix $t \geq 0$.  We show that the map $(\la_1, \ldots, \la_n) \mapsto  (\la_1, \ldots, \la_{n-1})$
is a bijection from the set of elements $\la \in L_n$  satisfying $\deg(\lambda)=t$
to $tP_{n-1} \cap \integers^{n-1}$.
Observe first that if $\lambda$ is a lecture hall partition, then $\lambda_n - \lambda_{n-1}= t$ implies
\[
\frac{ \lambda_{n-1} }{ n-1 } \ \le \ \frac { \lambda_n } n \ = \ \frac{ \lambda_{n-1} + t } n
\]
which in turn simplifies to
\[
\frac { \lambda_{n-1} }{ n-1 } \ \le \ t \, .
\]
It is immediate that our map is injective, since the preimage of a point in $tP_{n-1} \cap \integers^{n-1}$ is uniquely determined by adding $t$ to $\lambda_{n-1}$.
To see that the map is surjective, suppose that $(\la_1, \ldots, \la_{n-1}) \in tP_{n-1} \cap \integers^{n-1}$ and consider the partition
\[
(\la_1, \ldots, \la_{n-1}, \la_{n-1}+t) \, .
\]
This is an element of $L_n$ of degree $t$, which is verified by observing that
\[
\frac{\la_{n-1}+t}{n} \ \ge \ \frac{\la_{n-1}}{n}+\frac{1}{n}\frac{\la_{n-1}}{(n-1)} \ = \ \frac{(n-1)\la_{n-1}+\la_{n-1}}{n(n-1)} \ = \ \frac{n\la_{n-1}}{n(n-1)} \ = \ \frac{\la_{n-1}}{n-1} \, ,
\]
and hence our map is surjective.

Thus, 
\[\sum_{\lambda\in L_n\cap \integers^n} z^{\deg(\lambda)}
\ = \ \sum_{t \geq 0}i(P_{n-1},t) \, z^t
\ = \ \frac{A_{n-1}(z)}{(1-z)^{n}} \, . \qedhere
\]
\end{proof}

Throughout the remainder of this note we represent a lecture hall partition $\lambda$ as either the row vector $(\lambda_1,\ldots,\lambda_n)$ or as a column vector 
\[
\left[
\begin{array}{c}
\lambda_n \\
\lambda_{n-1}\\
\vdots \\
\lambda_1
\end{array}
\right] .
\]
Theorem~\ref{thm:lecthalldes} says that $A_{n-1}(z)$ is the $h^*$-polynomial of the $(n-1)$-dimensional simplex obtained as the intersection of $L_n$ with the hyperplane $\lambda_n-\lambda_{n-1}=1$.
Because the columns of the matrix below are the minimal ray generators for $L_n$ and they each lie on the hyperplane $\lambda_n-\lambda_{n-1}=1$, these columns must form the vertices of the intersection, which we denote
\[
Q_n \ := \ \mathrm{conv}\left[
\begin{array}{ccccc}
1 & n & n & \cdots & n \\
0 & n-1 & n-1 & \cdots & n-1 \\
0 & 0 & n-2 & \cdots & n-2 \\
\vdots & \vdots &  & \ddots & \vdots \\
0 & 0 & 0 & \cdots & 1 \\
\end{array} 
\right] \, .
\]
We will establish two geometric properties of $Q_n$.

Applying to $Q_n$ the unimodular transformation that takes consecutive row differences produces the polytope
\[
R_n \ := \ \mathrm{conv}\left[
\begin{array}{ccccc}
1 & 1 & 1 & \cdots & 1 \\
0 & n-1 & 1 & \cdots & 1 \\
0 & 0 & n-2 & \cdots & 1 \\
\vdots & \vdots &  & \ddots & \vdots \\
0 & 0 & 0 & \cdots & 1 \\
\end{array}
\right] .
\]
Note that $R_n$ is an $(n-1)$-dimensional simplex embedded at height $1$ in $\reals^n$.
The unimodular transformation above shows that $L_n$ is unimodularly equivalent to $\cone(R_n)$---this
means that $\cone(R_n)$ is the image of $L_n$ under some mapping in $\mathrm{SL}_n(\integers)$--- and hence we can freely use either presentation of this cone with respect to lattice-point enumeration.

Among the vertices for $R_n$, the right-most column in the defining matrix is the apex of a height-1 lattice pyramid over the convex hull of the remaining $n-1$ columns, which we denote by $\tilde{R}_n$.
Theorem~\ref{thm:lecthalldes} can be used to show that $\tilde{R}_n$ satisfies the following condition.

\begin{theorem}
The polytope $\tilde{R}_n$ is a lattice translate of a reflexive polytope.
\end{theorem}

\begin{proof}
Theorem~\ref{thm:lecthalldes} implies that the $h^*$-vector of $R_n$ is given by $A_{n-1}(z)$.
It is not hard to see that for a height-1 lattice pyramid over a polytope $P$, the $h^*$-polynomial of the pyramid is the same as that of $P$~\cite[Theorem 2.4]{BeckRobinsCCD}.  
By the aforementioned theorem of Hibi, since $A_{n-1}(z)$ has symmetric coefficients and is of degree $n-2$, and $R_n$ is a lattice pyramid over the $(n-2)$-dimensional polytope $\tilde{R}_n$, the result follows.
\end{proof}


\section{Triangulating the Lecture Hall Cone}\label{sec:triangulation}

In this section we prove that the polytopes $R_n$ admit regular, flag, unimodular triangulations.
We can represent $R_n$ as the convex hull of
\[
\left[
\begin{array}{cccc}
1 & 1 & \cdots & 1 \\
1 &  1 & \cdots & 1 \\
0 &  n-2 & \cdots & 1 \\
\vdots &  & \ddots & \vdots \\
0 &  0 & \cdots & 1 \\
\end{array}
\right], 
\left[
\begin{array}{c}
1 \\
0 \\
0 \\
\vdots \\
0 \\
\end{array}
\right],
\, \text{and} \, 
\left[
\begin{array}{c}
1 \\
n-1 \\
0 \\
\vdots \\
0 \\
\end{array}
\right]\, .
\]
Thus, we see that the intersection of $R_{n}$ with the hyperplane $\lambda_{n-1}=1$, which arises as the convex hull of the matrix on the left above, is a sub-polytope of $R_n$ that is an embedded copy of $R_{n-1}$.
Thus, $R_n$ is a union of two pyramids, with different heights, over $R_{n-1}$, 
pictured in Figure~\ref{fig:twopyramids}.

\begin{figure}
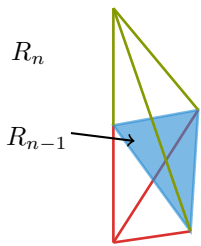

  \begin{polyhedron}{dim=3,phi=60,theta=70}
  \point{(0,3,0)}{A}
  \point{(5,3,0)}{B}
  \point{(0,5,0)}{C}
  \point{(0,3,3)}{D}
  \point{(0,3,-3)}{E}
  \edge{points={B,E},status=alert}
  \edge{points={A,E},status=alert}
  \edge{points={C,E},status=alert}
  \polygon{points={A,B,C},color=solarized-blue, opacity=0.6}
  \edge{points={B,D},status=focus}
  \edge{points={A,D},status=focus}
  \edge{points={C,D},status=focus}
  
  \node [label={[anchor=north]below: $R_n$}] at (0,1,3) {};
  
  \node at (0,1.2,0) { $R_{n-1}$};
\draw[thick,->,bend left=45] (0,2,0) -- (0,3.5,-0.5);
  \end{polyhedron}
  \caption{$R_n$ is a union of two pyramids over $R_{n-1}$.}
  \label{fig:twopyramids}
\end{figure}

Given a lattice polytope $S\subset \reals^n$, consider two integral linear functionals $\bl$ and $\bu$, where $\bl$ and $\bu$ have integer coefficients, such that $\bl\leq \bu$ on $S$.
Define the \emph{chimney polytope}
\[
\chim(S,\bl,\bu) \ := \ \left\{ (y,x)\in \reals\times \reals^n : \, x\in S, \ \bl(x)\leq y\leq \bu(x) \right\} .
\]

\begin{theorem}[Haase, Paffenholz, Piechnik, Santos \cite{regulartriangulations}]\label{thm:chimney}
If $S$ admits a regular, flag, unimodular triangulation, then so does $\chim(S,\bl,\bu)$.
\end{theorem}

\begin{theorem}\label{thm:lectchimney}
For all $n\geq 1$, the polytope $R_n$ admits a regular, flag, unimodular triangulation.
\end{theorem}


\begin{proof}
Let $\hat{1}$ denote the constant function with value $1$.
We first observe that $R_n$ is the union of
\[
\chim \left( R_{n-1}, \hat{1}, (n-1)\lambda_{n-1}-\sum_{i=1}^{n-2}\lambda_i \right)
\]
with the height-1 lattice pyramid formed from the convex hull of $(0,\ldots,0,1)$ and the face of the chimney polytope above obtained from the lower linear functional $\hat{1}$.
To see this, recall from our previous discussion that $R_{n-1}$ arises as the $(\lambda_{n-1}=1)$-slice of $R_n$ in $\reals^n$.
The above chimney polytope is the convex hull of the embedded copy of $R_{n-1}$ and the height-$(n-1)$ pyramid
point, as is verified by observing that the upper linear functional on $R_{n-1}$ agrees with the
$\lambda_{n-1}$-value of the vertices of $R_n$ with $(n-1)$st coordinate strictly greater than 0.
That the remainder of $R_n$ is contained, as claimed, in the height-1 lattice pyramid is clear from our description of $R_n$ relative to the embedded copy of $R_{n-1}$.
See Figure~\ref{fig:triangulation}.

Second, we use the fact that $R_1$ is a unit lattice segment as a base for induction to establish the theorem.
Assuming that $R_{n-1}$ admits a regular, flag, unimodular triangulation, Theorem~\ref{thm:chimney} implies the same for $\chim(R_{n-1}, \hat{1}, (n-1)\lambda_{n-1}-\sum_{i=1}^{n-2}\lambda_i)$.
This triangulation restricts to a regular, flag, unimodular triangulation of the face defined by the lower linear functional $\hat{1}$, and taking a unimodular lattice pyramid over this face extends this triangulation to all of $R_n$.
\end{proof}

\begin{figure}
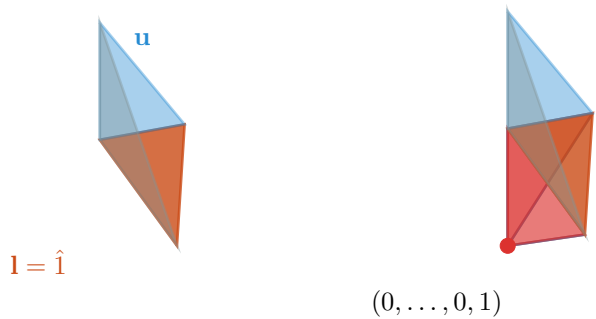

  \begin{polyhedron}{dim=3,phi=60,theta=70}
  \point{(0,3,0)}{A}
  \point{(5,3,0)}{B}
  \point{(0,5,0)}{C}
  \point{(0,3,3)}{D}
  \point{(0,3,-3)}{E}
  \polygon{points={A,B,C},color=solarized-orange,opacity=0.8} 
  \polygon{points={D,A,C},color=solarized-blue,opacity=0.4}  
  \polygon{points={D,A,B},color=solarized-base0,opacity=0.4}  
  \node [label={[anchor=north]below: {\color{solarized-blue}$\bu$}}] at (0,4,3) 
    {};
  \node [label={[anchor=north]right: { \color{solarized-orange} 
    $\bl=\hat{1}$}}] at (2,0.5,-1) {};
  \node [label={[anchor=north]right: }] at (0,-1,-3.5) {};
  \end{polyhedron}\qquad\qquad
  \begin{polyhedron}{dim=3,phi=60,theta=70}
  \point{(0,3,0)}{A}
  \point{(5,3,0)}{B}
  \point{(0,5,0)}{C}
  \point{(0,3,3)}{D}
  \point{(0,3,-3)}{E}
  \edge{points={B,E}}
  \edge{points={A,E}}
  \edge{points={C,E}}
  \polygon{points={A,B,E},color=solarized-red,opacity=0.4}
  \polygon{points={E,A,C},color=solarized-red,opacity=0.4}
  \polygon{points={B,A,E},color=solarized-red,opacity=0.4}
  \polygon{points={A,B,C},color=solarized-orange,opacity=0.8} 
  \polygon{points={D,A,C},color=solarized-blue,opacity=0.4}
  \polygon{points={B,A,D},color=solarized-base0,opacity=0.4}
  \vertex{point={E},color=solarized-red}
  \node [label={[anchor=north]right: { $(0,\ldots,0,1)$}}] at (0,1,-3.5) {};
  \end{polyhedron}
  \caption{The triangulation of the chimney polytope restricts to a regular, 
flag, unimodular triangulation of the face defined by $\hat{1}$. Taking a 
unimodular lattice pyramid over this face extends this triangulation to all of 
$R_n$.
}
  \label{fig:triangulation}
\end{figure}


\section{Hilbert Bases}\label{sec:hilbert}

Given a pointed rational cone $C\subset \reals^n$, the integer points $C\cap \integers^n$ form a semigroup.
It is known that such a semigroup has a unique minimal generating set called the \emph{Hilbert basis} of $C$
(see, e.g., \cite{millersturmfels}).
In this section we show that the cone $L_n$ has a (we think) interesting Hilbert basis.

To an element 
\[
A \ = \ \{i_1<i_2<\cdots<i_k\} \ \in \ 2^{[n-1]} \, ,
\]
i.e., a subset $A\subseteq [n-1]:=\{1,2,\ldots,n-1\}$, we associate the vector $v_A\in \integers^n$ defined by
\[
v_A \ := \ (0,0,\ldots,0,i_1,i_2,\ldots,i_k,i_k+1) \, ,
\]
where $v_\emptyset:=(0,0,\ldots,0,1)$.
The following lemmas will be useful.
\begin{lemma}
For integers $0<i$ and $0 \leq k \leq i$,  
\[\frac{k}{i} \ \le \ \frac{k+1}{i+1} \, .
\]
\begin{proof}
From the conditions,
\[
\frac{k+1}{i+1} - \frac{k}{i} \ = \ \frac{i(k+1)-k(i+1)}{i(i+1)} \ = \ \frac{i-k}{i(i+1)} \ \geq \ 0 \, . \qedhere
\]
\end{proof}
\label{lemma1}
\end{lemma}

\begin{lemma}
For positive integers $k, \ell,i $  if 
\[
\frac{k}{i} \ \le \ \frac{\ell}{i+1}
\]
then $k < \ell$.
\label{lemma2}
\end{lemma}
\begin{proof}
The conditions imply $\ell i \geq k(i+1) = ki+k > ki.$  So $\ell > k$.
\end{proof}
\begin{theorem}\label{thm:lseqhilbertbasis}
The Hilbert basis for $L_n$ is
\[
\mH_n \ := \ \left\{ v_A : \, A\in 2^{[n-1]} \right\} .
\]
\end{theorem}

\begin{remark}
There is a standard bijection between subsets of $n-1$ and compositions (i.e., ordered partitions) of $n$ given by sending $\{i_1<i_2<\cdots<i_k\}$ to $(n-i_k)+(i_k-i_{k-1})+\dots + (i_2-i_1)+i_1$.
This is precisely the unimodular transformation taking consecutive row differences previously used to convert between $Q_n$ and $R_n$.
For example, the subset $\{2,4,5\}\subset [6]$ bijects to the composition $(1,1,2,2)$, which is equivalent to our consecutive row difference transformation sending
\[
\left[
\begin{array}{c}
6 \\
5\\
4\\
2\\
0\\
0
\end{array}
\right] 
\mapsto
\left[
\begin{array}{c}
1 \\
1\\
2\\
2\\
0\\
0
\end{array}
\right] \, .
\]
\end{remark}

\begin{proof}[Proof of Theorem~\ref{thm:lseqhilbertbasis}]
 $\mH_n$ can be characterized as the set of 
 $(\la_1, \ldots, \la_n) \in \integers^n $ such that  $\la_n -\la_{n-1}=1$ and 
  for  some $j$,  $1 \leq j \leq n$,
 \[ 
 0 = \la_1 = \ldots = \la_{j-1} < \la_j < \ldots \la_{n-1} <\la_{n} \leq n.
 \]
We first show that $\mH_n\subset L_n$.  Let $\la \in \mH_n$.  
We need to check that $\frac{ \la_i }{ i } \le \frac{ \la_{i+1} }{ i+1 } $ for 
$1 \leq i <n$.  If $\la_i = 0$ the condition holds.
If $\la_i > 0$, then by  Lemma \ref{lemma1} and the characterization of $\mH_n$, 
\[
\frac{\la_i}{i} \ \le \ \frac{\la_i+1}{i+1} \ \le \ \frac{\la_{i+1}}{i+1}.
\]
We next show that
 $\mH_n$ contains all degree-1 elements in $L_n$ with respect to the grading $\lambda_n-\lambda_{n-1}=1$.
Let $\lambda$ be a degree-1 element in $L_n$.
Since $\lambda\in L_n$, 
\[
-n \, \lambda_{n-1}+(n-1)\lambda_n \ \ge \ 0 \, ,
\]
and substituting $\lambda_n-1=\lambda_{n-1}$ gives
\[
-n \, (\lambda_n-1)+(n-1)\lambda_n \ \ge \ 0 \, ,
\]
from which it follows that
\[
n \ \ge \ \lambda_n \, .
\]
To check that the other coordinates of $\la$ satisfy the constraints of $\mH_n$, note that if $\la_j > 0$ for
$1 \leq j < n-1$, then  since,  as a lecture hall partition, $\frac{ \la_j }{ j } \le \frac{ \la_{j+1} }{ j+1 } $, we conclude with Lemma \ref{lemma2} that $\la_j < \la_{j+1}$.
It is immediate that, since every $v_A$ has degree 1, no $v_A$ is a nonnegative integer combination of other elements in $\mH_n$.
Hence, our Hilbert basis must contain $\mH_n$.

To finish the proof, we show by induction on the degree that all elements in $L_n$ are nonnegative integer combinations of elements of $\mH_n$.  We have just shown that degree 1 elements of $L_n$ are in $\mH_n$.  Let $a \in L_n$ have degree $t>1$.  Let $j$ be the largest index such that $a_j < j$. 
If $j=n-1$ or $j=n$, write $a=b+c$ where $b=(0, \ldots,0,1)$ and
$c=(a_1, \ldots, a_{n-1}, a_n-1)$.  Then $b \in \mH_n$ and $c$ has degree $t-1$ and, by Lemma \ref{lemma1}, $c \in L_n$.

Otherwise, $j \leq n-2$ and  $a=b+c$ where
\[b= (a_1, \ldots,a_j,j+1,\ldots,n)\]
and
\[c=(0,0,\ldots,0,a_{j+1}-(j+1),a_{j+2}-(j+2),\ldots,a_n-n).\]
Then $b \in \mH_n$ and $c$ has degree $t-1$.  It remains to show $c \in L_n$.
As $a\in L_n$, we have for $1 \leq i < n$
\[
(i-1)a_i-i \, a_{i-1} \ \geq \ 0 \, .
\]
It follows that $c \in \ L_n$,  since for all $i\geq j$,
\[
(i-1)(a_i-i)-i(a_{i-1}-(i-1)) \ = \ (i-1)a_i-i \, a_{i-1} \ \geq \ 0 \, . \qedhere
\]
\end{proof}

\begin{remark}
We give here an alternative proof that all elements in $L_n$ are nonnegative integer combinations of elements of $\mH_n$ --- we chose to highlight the proof above because we feel it is independently interesting, as it explicitly relies on the number-theoretic structure of $L_n$.
The existence of the unimodular triangulation given by Theorem~\ref{thm:lectchimney} implies that the Hilbert basis for $L_n\cap \integers^n$ consists of precisely the elements of degree one.
From the bijection in the proof of Theorem~\ref{thm:lecthalldes}, we have that the elements of degree $t$ in $L_n$ are in bijection with the elements of $tP_{n-1} \cap \integers^{n-1}$.
Since $\left| tP_{n} \cap \integers^{n} \right|=(t+1)^{n}$, we have $\left|\left\{ \la \in L_n  : \deg(\la)=t \right\}\right|=(t+1)^{n-1}$.
Thus, the Hilbert basis $\mH_n$ contains (at most) $2^{n-1}$ elements, since $\left|\left\{ \la \in L_n  : \deg(\la)=1 \right\}\right|=2^{n-1}$.
As we have identified $2^{n-1}$ such elements, these must constitute all of $\mH_n$.
\end{remark}

\section{Triangulations and Unit Cubes}\label{sec:outlook}

In this section we briefly discuss consequences of the existence of regular, flag, unimodular triangulations of $Q_n$.
It is known~\cite{StanleyDecompositions} that for a lattice polytope $P$ with a unimodular triangulation, the triangulation has $\sum_{i=0}^{\dim(P)} h^*_i(P)$ maximal unimodular simplices.
Hence, Theorem~\ref{thm:lecthalldes} implies that any unimodular triangulation of $Q_n$ has $(n-1)!$ maximal unimodular simplices. 

There is another well-known polytope having $h^*$-polynomial given by $A_{n-1}(z)$ and admitting regular, flag, unimodular triangulations, namely the cube $[0,1]^{n-1}$.
An interesting connection between $L_n$ and $\cone([0,1]^{n-1})$ is that the integer points in both of them are generated over $\integers_{\geq 0}$ by degree-$1$ elements in bijection with subsets of $[n-1]$.
Namely, we can encode each subset
\[
  A \ = \ \{i_1<i_2<\cdots<i_k\}\ \subseteq \ [n-1]
\]
as a vector in $\integers^n$ in two ways.
First, let $(1,m^A):=(1,m^A_1,\ldots,m^A_{n-1})$ be defined by $m^A_i=1$ if $i\in A$, and $m^A_i=0$ otherwise.
Second, let $v_A\in \integers^n$ be defined as before.
Informally, for $(1,m^A)$ we encode $A$ by a characteristic vector, while for $v_A$ we encode $A$ directly by placing the elements of $A$ in increasing order, making the last entry one more than $\max(A)$, and padding the left-most entries with zeros as needed.
Note that, by definition, $(1,m^A)\in \cone([0,1]^{n-1})$ and $v_A\in L_n$ each have degree~$1$ in their respective gradings.
Thus, both $L_n$ and $\cone([0,1]^{n-1})$ can be viewed as cones generated by the subsets of $[n-1]$.
While their geometric and arithmetic structure are quite different, our results and Conjecture~\ref{conj:lecthallsperner} below indicate that these cones share surprising properties.

An important regular, flag, unimodular triangulation of $[0,1]^{n-1}$ is obtained by intersecting
$[0,1]^{n-1}$ with the real braid arrangement, 
i.e., the set of hyperplanes defined by $x_i=x_j$ for all $i\neq j$ between $1$ and $n-1$.
Recall that a \emph{Sperner $2$-pair} of $[n-1]$ is a pair of subsets $A,B\subset [n-1]$ such that neither $A$ nor $B$ is contained in each other. 
One important property of the braid triangulation of $[0,1]^{n-1}$ is that the number of minimal non-edges
equals the number of Sperner $2$-pairs of $[n-1]$.
(All of these properties are most easily seen by observing that $[0,1]^{n-1}$ is the order
polytope~\cite{StanleyTwoPosetPolytopes} for an $(n-1)$-element antichain, hence the maximal unimodular
simplices in this triangulation correspond to linear extensions of the antichain, i.e., chains in the Boolean
algebra $2^{[n-1]}$.)

Based on experimental evidence for $n\leq 7$, and the established connections between $L_n$ and $\cone([0,1]^{n-1})$, the following conjecture seems plausible.
\begin{conjecture}\label{conj:lecthallsperner}
There exists a regular, flag, unimodular triangulation of $R_n$ that admits a shelling order such that the maximal simplices of the triangulation are indexed by $\pi\in
S_{n-1}$ and each such simplex is attached along $\des(\pi)$ many of its facets.
Further, the number of minimal non-edges in the triangulation is the number of Sperner $2$-pairs of $[n-1]$.
\end{conjecture}
Conjecture~\ref{conj:lecthallsperner} has thus far proven to be a real challenge.
One might hope that a ``deformed'' version of the braid triangulation for $[0,1]^{n-1}$ could be imposed on $R_n$, though so far such a triangulation has proven to be elusive.

The second part of Conjecture~\ref{conj:lecthallsperner} can be reformulated in the language of toric algebra using Sturmfels' correspondence between regular, flag, unimodular triangulations and quadratic, squarefree Gr\"{o}bner bases for toric ideals~\cite{millersturmfels,sturmfels}.
Letting $I_n$ denote the toric ideal defining the semigroup algebra $\C[L_n\cap \integers^n]$, Theorem~\ref{thm:lectchimney} implies that there exists a term order for which $I_n$ has a quadratic, squarefree Gr\"{o}bner basis.
The second part of Conjecture~\ref{conj:lecthallsperner} is equivalent to the statement that there exists a reduced Gr\"{o}bner basis for $I_n$ containing the same number of elements as Sperner $2$-pairs of $[n-1]$.
When considering randomly tested term orders, with surprising frequency we have observed squarefree quadratic initial ideals with the desired number of Gr\"{o}bner basis elements.

\bibliographystyle{plain}
\bibliography{Braun}
 
\end{document}